\documentclass[10pt]{asl}
\usepackage[margin=1.7in]{geometry}               % See geometry.pdf to learn the layout options. There are lots.
\usepackage{graphicx}
\usepackage{amssymb}

\usepackage{epstopdf}
\usepackage{tikz}
\usetikzlibrary{arrows}
\usepackage{times}
\usepackage{ mathrsfs }
\usepackage{setspace}
\usepackage{stmaryrd}
\usepackage{mathabx}

\renewcommand{\leq}{\leqslant}

\newtheorem{Def}{Definition}
\newtheorem{Thm}{Theorem}

\newtheorem{Cor}[Thm]{Corollary}
\newtheorem{Lem}[Thm]{Lemma}
\newtheorem{Pro}[Thm]{Proposition}

\address{Department of Knowledge-Based Mathematical Systems, Johannes Kepler Universit\"at, Austria}

\usepackage{color}

\DeclareMathSymbol{\lsim}{\mathord}{symbols}{"18}

\allowdisplaybreaks
\DeclareMathSymbol{\lsim}{\mathord}{symbols}{"18}
\title{Infinitary propositional relevant languages with absurdity}
\author{Guillermo Badia}
\date{}                                           % Activate to display a given date or no date

\begin{document}
\maketitle
\begin{abstract} Analogues of Scott's isomorphism theorem, Karp's theorem as well as  results on lack of compactness and strong completeness are established for infinitary propositional relevant logics. An ``interpolation theorem" (of a particular sort introduced by Barwise and van Benthem)  for the infinitary quantificational boolean logic $L_{\infty \omega}$ holds. This yields a preservation result characterizing the expressive power of infinitary relevant languages with absurdity using the model-theoretic relation of relevant directed bisimulation as well as a Beth definability property.
\bigskip


  \end{abstract}

\section{Introduction}\label{sec:int}

In these pages we explore the model theory of a twofold non-classical logic:  infinitary relevant propositional  logic. By extending the language of relevant logic by adding infinitary conjunctions and disjunctions, we naturally gain some expressive power.  Such extensions have been toyed with from time to time in the context of relevant logic in an unsystematic and informal way (cf. \cite{rou, fine, fine2}). In \cite{rou} (p. 336), Routley reports some unpublished (and, according to him, not overly successful) attempts to study infinitary relevant logic. 

We will be working  in the well-known Routley-Meyer semantics  \cite{rm, rm1, rm2, r, dunn}.  This is the more or less standard non-algebraic semantics for relevant logic (\cite{thomas, robles} are examples of quite recent applications). The reader can find a survey of the alternatives in \cite{dunn}, though.\footnote{\cite{mares2} is a recent contender for the quantificational case, where incompleteness had been found by Fine \cite{fine2}.}

 Though the heyday of infinitary logic seems to be long gone, important results remain. In the next sections, we will  obtain relevant analogues of some of them such as Karp's theorem or Scott's isomorphism theorem. Karp's theorem (Corollary  3.5.3 in \cite{hodges}) is the claim that for any two models, $L_{\infty \omega}$-equivalence is the same as the existence of a family of partial isomorphisms with the back and forth properties.  Scott's isomorphism theorem (Corollary 3.5.4 in \cite{hodges}) says that, for denumerable models, making a single special formula true suffices to characterize a structure up to isomorphism.

The main problem we will solve here, though,  is that of characterizing the expressive power of infinitary relevant logic. This will be accomplished by establishing a generalized interpolation result for the classical infinitary logic $L_{\infty \omega}$, from which the desired characterization will follow in the form of a preservation theorem involving relevant directed bisimulations. On a historical note,  directed bisimulations were introduced in \cite{kurto} and though it was hinted there, it seems like \cite{restall} is the first time they were applied to the study of substructural  logics in print. Recently, they have been shown to have a fundamental place in the model theory of relevant logic in the Routley-Meyer semantic framework (cf. \cite{ba}, where the finitary case has been studied) analogous to bisimulations in the Kripke semantics for modal logic.

The results on expressive power in this paper can be seen as a continuation of the work in \cite{ba}, turning our attention this time to the realm of infinitary languages. There are certain differences in method worth mentioning, though. In  \cite{ba}, there was an appeal to the machinery of saturated models in order to establish a preservation theorem characterizing relevant formulas as a fragment of first order logic. This was, in fact, unnecessary for a much direct proof through a simple application of the compactness theorem of first order logic was possible. It would have simply require to introduce the notion of a relevant directed $n$-bisimulation, a finite approximation of a relevant directed bisimulation.\footnote{Incidentally, this is how the main result of \cite{ba2} characterizing the expressivity of propositional bi-intuitionistic languages was obtained.} This approach is so basic that generalizes to logics having some minimal forms of compactness such as $L_{\infty \omega}$. That is the main motivation behind our introduction of 
relevant directed $\alpha$-bisimulations in Definition \ref{def:ndbt}.

 In \S \ref{sec:rm}, we introduce the Routley-Meyer semantics for infinitary propositional relevant languages with absurdity. In \S \ref{sec:inc}, we show that infinitary relevant languages with absurdity are, in general, lacking compactness and most reasonable formal systems based on them are not strongly complete.  In \S \ref{sec:db}, we define relevant directed bisimulations establishing some basic propositions, including a relevant Karp theorem while in \S \ref{sec:s}, we prove a relevant analogue of Scott's isomorphism theorem.  In \S \ref{sec:pre}, we prove an interpolation theorem for the infinitary quantificational boolean logic $L_{\infty \omega}$ which implies  a preservation theorem saying that the formulas of $L_{\infty \omega}$  preserved under relevant directed bisimulations are exactly infinitary relevant formulas, as well as a Beth definability result. Finally, in \S \ref{sec:con} we briefly summarize our work.

\section{Routley-Meyer Semantics}\label{sec:rm}

In this section, we will review the Routley-Meyer semantics for propositional infinitary relevant languages with absurdity and their embeddability  in more traditional infinitary languages. 

 Let $\kappa$ be some infinite cardinal. An \emph{infinitary relevant language with absurdity}  $L_{\kappa \omega}^{\rightarrow}$ contains  a possibly finite list $\mathtt{PROP}$  of propositional variables $p, q, r \dots$  and  the logical symbols $\bot$ (an absurdity constant), $\lsim$ (negation), $\bigwedge$ (conjunction), $\bigvee$ (disjunction) and $\rightarrow$ (implication). Formulas are constructed as expected:

\begin{center}
$\phi ::= p \ | \ \bot \ | \ \lsim \phi \ | \  \bigwedge_{i \in I} \phi_i \ |  \ \bigvee_{i \in I} \phi_i \ |\ \phi \rightarrow \psi \ $,
\end{center}
where $p \in \mathtt{PROP}$ and $|I| \ \textless  \ \kappa$.  The infinitary relevant language with absurdity  $L_{\infty \omega}^{\rightarrow}$ comes from letting the index $I$ of a disjunction or a conjunction take any cardinality whatsoever.  $L_{\omega \omega}^{\rightarrow}$ is just an ordinary finitary relevant language.

A comment on the presence of $\bot$ in our languages is in place here, given that $\bot$ is not standardly part of the languages of relevant logic (cf. \cite{an}). The results in these pages cannot dispense with $\bot$, since languages without $\bot$ have no reasonable model-theoretic characterization. The interested reader is advised to consult \S 4 in \cite{ba}. 

 Note that  implications  are still finitary in the sense that we can only build formulas of the form

\begin{center}
$\phi_0 \rightarrow (\phi_1 \rightarrow (\phi_2 \rightarrow ( \dots \rightarrow \phi_{\lambda})\dots)$
\end{center}
when $\lambda$ is finite. This is the reason for writing $\omega$ in $L_{\kappa \omega}^{\rightarrow}$, it basically bounds the possible number of iterations of a $\rightarrow$ symbol in a formula. This notation should not be confused with the classical notation where the second subscript  is used to bound the possible length of a string of quantifiers.\footnote{It is opaque whether there is a connection here. For instance, $\underbrace{ \phi \rightarrow (\phi \rightarrow (\dots (\phi}_{\omega-\phi s} \rightarrow \psi) \dots))$ could be na\"ively translated $-$without the intervention of infinitely long strings of quantifiers$-$ into a ``classical infinitary" language with the appropriate signature, using the translation function given below, as $\forall y_0z_0 (Rxy_0z_0 \wedge T_x(\phi)^{y_0/x} \supset \forall y_1, z_1 (Rz_0y_1z_1 \wedge T_x(\phi)^{y_1/x} \supset (\dots \forall vu(Rz_{\omega}vu \wedge T_x(\phi)^{v/x} \supset T_x(\psi)^{u/x}) \dots)))$. The problem is that this is not a formula of any classical infinitary language $L_{\kappa \lambda}$. The reason is that it violates the well-foundedness of the subformula relation (Lemma 1.3.3 from \cite{dick}). To see this note that the collection of formulas $\forall y_iz_i (Rz_{i-1}y_iz_i \wedge T_x(\phi)^{y_i/x} \supset \forall y_{i+1}, z_{i+1} (Rz_iy_{i+1}z_{i+1} \wedge T_x(\phi)^{y_{i+1}/x} \supset (\dots \forall vu(Rz_{\omega}vu \wedge T_x(\phi)^{v/x} \supset T_x(\psi)^{u/x}) \dots)))$ ($0<i<\omega$) has no minimal element according to the subformula relation.}

An example of a connective definable in $L_{\infty \omega}^{\rightarrow}$ (but not in $L_{\omega \omega}^{\rightarrow}$) is $\xrightarrow{\omega}$ (iterated entailment), which was introduced by Humberstone (see \cite{bradyu}, p. 36). The formula $\phi \xrightarrow{\omega} \psi$ means that for some natural number $n > 1$, 
\begin{center}
$\underbrace{ \phi \rightarrow (\phi \rightarrow (\dots (\phi}_{n-\phi s} \rightarrow \psi) \dots))$
\end{center}
  holds. This, of course, boils down to an infinitary disjunction of finitary implications:
\begin{center}
$\bigvee_{n>1}\underbrace{ \phi \rightarrow (\phi \rightarrow (\dots (\phi}_{n-\phi s} \rightarrow \psi) \dots))$.
\end{center}

As we announced in $\S$\ref{sec:int}, we will be working in the Routley-Meyer semantic framework. In this setting, a  \emph{model} for $L_{\kappa \omega}^{\rightarrow}$ will be a structure $M = \langle W, R, *, T,  V\rangle$, where $W$ is a non-empty set, $T \in W$, $*$ is an  operation $*: W \longrightarrow W$ (the so called Routley star), $R \subseteq W \times W \times W$ and $V$ is a valuation function $V : \mathtt{PROP}  \longrightarrow \wp(W)$. In what follows we frequently omit $T$ from the presentation of our models since nothing essential hinges on that (given that we will not be considering any connectives involving $T$ in its semantics) and the reader can easily fill in the omitted details.

We define satisfaction at $w$ in $M$ recursively as follows:

\begin{center}
%\begin{table}[h]
\begin{tabular}{llll}
$M, w \Vdash \bot$ & never&                                 \\
$M, w \Vdash p$ & iff &  $w \in V( p)$                                 \\
                                                           $M, w \Vdash (\lsim \phi)$           & iff &   $M, w^{*} \nVdash \phi$ \\
                                                             $M, w \Vdash  (\bigwedge_{i \in I} \phi_i)$        &   iff    &      $M, w \Vdash  \phi_i$  for every $i \in I$.    \\
 $M, w \Vdash  (\bigvee_{i \in I} \phi_i)$        &   iff    &      $M, w \Vdash  \phi_i$  for some $i \in I$.    \\
$M, w \Vdash \phi \rightarrow \psi$           & iff &   for every $a, b$ such that $Rwab$, \\                                                                                                                                 
& &    if $M, a \Vdash \phi$ then $M, b \Vdash \psi$. \\                                                                                                                                 
\end{tabular}
%\end{table}
 \end{center}
Note that as $\bot$ gives us a means to define the empty class of models, $\top =_{df} (\lsim \bot)$ allows  defining the class of all models since it is invariably true (for recall that  $\bot$  invariably fails at $w^*$ for any $w$). 

The basic semantic units in relevant logic are (as in modal logic) \emph{pointed models}, that is, pairs $(M, w)$ where $w$ is some distinguished element of $W$. This is simply due to the fact that formulas are evaluated locally, at worlds.

 By considering restricted classes of Routley-Meyer structures where the relation $R$ has certain properties and only some valuations are admitted, we can get classes of models corresponding to a number of formal systems of relevant logic like ${\bf B}, {\bf T}$ or ${\bf R}$. Next we will consider some famous examples from \cite{r}.

Consider a relevant language  with absurdity $L$. A structure $ \langle W, R, *,  T, V\rangle$ is called a \emph{${\bf B}$-model} if for any $x, y, z, v \in W$:
\begin{itemize}
\item[(i)] $RTxx$  
\item[(ii)] $RTxv$ and $Rvyz$ implies that $Rxyz$
\item[(iii)]$ x = x^{**}$
\item[(iv)] $RTxy$ only if $RTy^* x^*$.
\item[(v)] $x \in V(p)$ and $RTxy$ implies that $y \in V(p)$. 

\end{itemize}
An \emph{${\bf R}$-model} is  a  ${\bf B}$-model where condition (iv) is strengthened to
\begin{itemize}
\item [(iv)$^{\prime}$]$Rzxy$ only if $Rzy^*  x^*$,

\end{itemize}
and, furthermore (abbreviating the claim that there is a $u$ such that $Rxyu$ and $Ruzv$ as $R^2(xy)zv$, and the claim that  there is an $u$ such that $Rxuv$ and $Ryzu$ as  $R^2x(yz)v$), for any $x, y, z, v \in W$:
\begin{itemize}
\item[(v)] $R^2(xy)zv$ only if $R^2x(yz)v$ 
\item[(vi)] $Rxxx$
\item[(vii)] $Rxyz$ only if $Ryxz$.
\end{itemize}
 An  \emph{${\bf RM}$-model}  is an ${\bf R}$-model such that 
and for any $x, y, z \in W$:
\begin{itemize}
\item[(v)] $Rxyz$ only if either $RTxz$ or $RTyz$.
\end{itemize}

When $\vdash$ is the deducibility relation of some formal system {\bf S} of relevant logic, a syntactic claim of the form $\phi \vdash \psi$ is to be interpreted on the class of corresponding models $V_{\bf S}$ as saying that  $M, T \Vdash \phi$ only if $M, T \Vdash \psi$ for every model $M \in V_{\bf S}$. In what follows we will use the symbol $V_{\bf S}$ as a variable for the class of models corresponding to any system ${\bf S}$ described in \cite{r} between ${\bf B}$ and ${\bf RM}$.\footnote{A caveat is in place here. The \emph{variable sharing property} is a folklore requirement from any formal system of relevant logic. The property states that whenever $\phi \rightarrow \psi$ is a theorem then $\phi$ and $\psi$ must share some propositional variable in common. When our language has $\bot$, the principle fails quite easily since $\bot \rightarrow \theta$ (for arbitrary $\theta$) would be a theorem, tempting one to claim that no system involving $\bot$ should qualify as a system of relevant logic. However, Yang \cite{ya} has suggested recently the \emph{strong implicit relevance property} as a nice substitute of the variable sharing property that would allow for systems containing $\bot$.}

 Next we give an example of the increased expressive power of infinitary relevant languages. Suppose $\Phi$ and $\Psi$ are  sets of formulas.  We speak of the pair $(\Phi, \Psi)$ as being \emph{satisfiable} or \emph{having a model} in a class $K$  of pointed  models if there is a model $(M, w) \in K$ such that $M, w \Vdash \phi$ for each $\phi \in \Phi$ and  $M, w \nVdash \psi$ for every $\psi \in \Psi$. These pairs are called \emph{tableaux} in \cite{chagrov} (pp.37-38).\footnote{See also the bi-theories in \cite{restall2}.} Let $V$ be a class of pointed models. A class of pointed  models $K \subseteq V$ is said to be  \emph{axiomatizable} in $L_{\omega \omega}^{\rightarrow}$ with respect to $V$  if there is a set of formulas $\Gamma$ of $L_{\omega \omega}^{\rightarrow}$ such that $K = \mbox{\emph{Mod}}(\Gamma)$ $-$where  $\mbox{\emph{Mod}}(\Gamma)$ the class of pointed model satisfying $\Gamma$.  Let $(M, w)$ be a model for $L_{\omega \omega}^{\rightarrow}$.  We say that $(M, w)$ is inconsistent if for some $p \in  \mathtt{PROP} $, $M, w \Vdash (p \wedge (\lsim p ))$.

Inconsistency is definable by a sentence of a propositional relevant language with absurdity $L_{\omega \omega}^{\rightarrow}$ if  $\mathtt{PROP} $ is finite, for in this case $\bigvee_{p \in \mathtt{PROP}} (p \wedge (\lsim p ))$ expresses that a model is inconsistent. If the signature is not finite, inconsistency is not in general a property axiomatizable  in $L_{\omega \omega}^{\rightarrow}$. This has been pointed out for $LP$ essentially in \cite{ferguson} with an argument using  a version of \L o\'s's theorem on ultraproducts.

\begin{Pro}\label{pro:inc} If  $|\mathtt{PROP}|  \geqslant \omega$, inconsistency is not a property of models axiomatizable in $L_{\omega \omega}^{\rightarrow}$ with respect to any $V_{\bf S}$. \end{Pro}

\begin{proof} Suppose it were.  Say  the theory $\Theta$ axiomatizes the class of inconsistent models. Now, the pair $(\Theta, \Phi)$ where $\Phi = \{ (p \wedge (\lsim p )) : p \in \mathtt{PROP}\}$ is finitely satisfiable in $V_{\bf S}$. To see this take a finite subset $\{p_0, \dots , p_n\} \subset  \mathtt{PROP}$. Consider the model $\langle W, R, *, V, T\rangle$ (in $ V_{\bf S}$ since it is in  $ V_{\bf RM}$) such that 

\begin{center}
\begin{itemize}
\item[] $W=\{t, s\}$ 
\item[] $* = \{\langle t, s \rangle, \langle s, t \rangle\}$
\item[] $R = \{\langle t,t,t\rangle, \langle t,s,t\rangle, \langle t, s, s\rangle, \langle s,t,t\rangle, \langle s,t,s\rangle, \langle s, s,t\rangle, \langle s, s, s\rangle\}$
\item[] $T=t$
\item[] $V(p_i) = W $  (for $i = 0, \dots, n$)
\item[] $V(q) = \{t\}$ (for $q \in \mathtt{PROP}$, and $q \neq p_i$ for $i = 0, \dots, n$)
\end{itemize}
\end{center}
We see that if  $q \in \mathtt{PROP}$ but $q \neq p_i$ for $i = 0, \dots, n$, then $M, t \Vdash (q \wedge (\lsim q )) $. On the other hand, $M, t \nVdash (p_i \wedge (\lsim p_i ))$ $(i = 0, \dots, n) $ since $t^* = s \in V(p_i)$, which means that $M, t \nVdash (\lsim p_i)$.  

Finally, by Proposition 2.5 of \cite{ba}, the pair $(\Theta, \Phi)$ is satisfiable in $V_{\bf S}$, which is a contradiction since by definition $\Theta$ says that at least one of $\phi \in \Phi$ must hold.
\end{proof}

When $|\mathtt{PROP}|  \geqslant \omega$, inconsistency is expressible by a single formula in the extension   $L_{|\mathtt{PROP}|^+ \omega}^{\rightarrow}$ of $L_{\omega \omega}^{\rightarrow}$. Again, $\bigvee_{p \in \mathtt{PROP}} (p \wedge (\lsim p ))$ expresses that  a model is inconsistent. This fact shows that $L_{|\mathtt{PROP}|^+ \omega}^{\rightarrow}$ is a proper expressive extension of $L_{\omega \omega}^{\rightarrow}$.

Consider an infinitary language with equality and boolean negation admitting conjunctions and disjunctions of size at most $\kappa$ (the standard reference for the study of such laguages is  \cite{dick}) and quantifications over at most finitely many variables that comes with an individual constant symbol $T$, one function symbol $*$, a distinguished three place relation symbol $R$,  and a unary predicate $P$ for each $p \in \mathtt{PROP}$. Following the tradition in modal logic, we might call this a \emph{ correspondence language} $L^{corr}_{\kappa \omega}$ for   $L_{\kappa \omega}^{\rightarrow}$ (cf. \cite{blackburn}). Now we can read a model $M$ as a classical model for $L^{corr}_{\kappa \omega}$ in a straightforward way:   $W$ is taken as the domain of the structure, the constant $T$ denotes the obvious distinguished world, $V$ specifies the denotation of each of the predicates $P, Q,  \dots$, while $*$ is the denotation of the function symbol $*$ of $L^{corr}$, and $R$ the denotation of the relation $R$ of $L^{corr}_{\kappa \omega}$.

Where $t$ is a term in the correspondence language, we write $\phi^{t/x}$ for the result of replacing $x$ with $t$ everywhere in the formula $\phi$. As expected, it is easy to specify a translation from the formulas of the basic relevant language with absurdity to the correspondence language as follows:

\begin{center}
%\begin{table}[h]
\begin{tabular}{lll}
\,\,\,\,\,\,\,\,\,\, \,\,\, \,\,\, \,\,\, $T_x(\bot)  $ & = &  $ \neg Rxxx \wedge Rxxx$                                 \\
\,\,\,\,\,\,\,\,\,\, \,\,\, \,\,\, \,\,\, $T_x(p)  $ & = &  $ Px$                                 \\
\,\,\,\,\,\,\,\,\,\, \,\,\, \,\,\,$T_x(\lsim \phi)$ & =  &  $ \neg T_x(\phi)^{x^{*}/x}$     \\
\,\,\,\,\,\,\,\,\,\, \, $T_x(\bigwedge_{i \in I}\phi_i) $           & =  &  $\bigwedge_{i \in I}  T_x(\phi_i)$
 \\  
\,\,\,\,\,\,\,\,\,\, \,  $T_x(\bigvee_{i \in I}\phi_i) $           & =  &  $\bigvee_{i \in I}  T_x(\phi_i)$
 \\                                                                                                                                  
                                   \,\,\,\,\,\,\,\,\,\, \  $T_x(\phi \rightarrow \psi)  $                            &=  &  $\forall y, z (Rxyz \wedge T_x(\phi)^{y/x} \supset T_x(\psi)^{z/x})$.  \\

\end{tabular}
%\end{table}
 \end{center}
The symbols $\neg$ and $\supset$ appear here representing  boolean negation and material implication in quantificational infinitary logic (which should not be confused  with the relevant $\lsim$ and $\rightarrow$).

The following proposition gives a bridge between the satisfaction relation $\Vdash$ for relevant propositional languages we just defined and the standard satisfaction relation $\vDash$ from classical logic (where when $\phi$ is a classical formula, we write $M \vDash \phi [w]$ to mean that the object $w$ satisfies $\phi$ in the usual Tarskian sense).

\begin{Pro} \label{pro:tr}  For any $w$, $M, w \Vdash \phi$ if and only if  $M \vDash T_x(\phi) [w]$. \end{Pro}

\begin{proof} We simply need to note that, according to the Routley-Meyer semantics, each propositional  relevant formula $\phi$ says the same about $w$ as $T_x(\phi)$ does in the Tarskian semantics. 

 \end{proof}

The existence of a satisfaction preserving translation function allows us to study relevant languages with absurdity as fragments of model-theoretically better  understood creatures.

\section{ Failure of compactness and strong completeness}\label{sec:inc}

In this section, we study briefly a phenomenon pervasive in infinitary logic even at the propositional level: the loss of compactness. This quickly leads to a loss of strong completeness for any reasonable infinitary formal system (cf. \cite{karp}). Such seems to be the price to pay for having infinitely long conjunctions and disjunctions around. Here we will focus our attention on specific classes of models since we will be discussing questions sensitive to the choice of formal system such as incompleteness. 

\begin{Def} Let $L_{\kappa \omega}^{\rightarrow}$ be a relevant language with absurdity, $K$ a class of Routley-Meyer structures for it and $(\Phi, \Psi)$ a pair of collections of relevant formulas. $L_{\kappa \omega}^{\rightarrow}$  is said to be \emph{$\lambda$-compact with respect to K} if  for every  $\Phi_0 \subseteq \Phi$ and $\Psi_0 \subseteq \Psi$ such that $|\Phi_0|, |\Psi_0| \ \textless \ \lambda$, the pair $(\Phi_0, \Psi_0)$  has a model in $K$ only if $(\Phi, \Psi)$  has a model in $K$. \end{Def}

\begin{Pro} \label{pro:in}  Let $|\mathtt{PROP}|  \geqslant \kappa$.  $L_{\kappa \omega}^{\rightarrow}$  is $\kappa$-compact with respect to some $V_{\bf S}$  only if $\kappa$ is a regular limit cardinal. \end{Pro}

\begin{proof}  Suppose $\kappa$ is a sucessor cardinal $\xi +1$. Without loss of generality, assume $\mathtt{PROP}$ is composed of double indexed propositional variables $p_{\lambda \gamma}$ ($\lambda \ \textless \ \xi+1, \gamma \ \textless \ \xi$). Consider the set of formulas
\begin{center}
$\Delta = \{\bigvee_{\gamma \ \textless \ \xi} p_{\lambda \gamma }: \lambda \ \textless \ \xi+1\} \cup \{p_{ \lambda \gamma } \wedge p_{ \mu \gamma} \rightarrow \bot :\mu \neq\lambda, \mu, \lambda \ \textless \ \xi+1, \gamma \ \textless \ \xi \}$.
\end{center}
Take any $\Delta_0 \subset \Delta$ such that $|\Delta_0| \leq \xi$. By the axiom of choice, there is a one-to-one mapping $f$ from the set of all $\lambda$ such that $p_{\lambda \gamma}$ for some $\gamma$ appears in a formula of $\Delta_0$ into $\xi$.  We build the  model where $W=\{t\}$, $R =\{\langle t, t, t \rangle\}, *=\{\langle t, t \rangle\}$, and we define $V$ as follows: $V(p_{\lambda f(\lambda)}) = W$, and $V(p_{\lambda \gamma}) = \emptyset$ if $\gamma \neq f(\lambda)$. It is clear that $M, t \Vdash \bigvee_{\gamma \ \textless \ \xi} p_{\lambda \gamma }$ for all disjunctions in $\Delta_0$ with $\gamma \ \textless \ \xi+1$. Now take any $p_{ \lambda \gamma } \wedge p_{ \mu \gamma} \rightarrow \bot \in \Delta_0$ such that $ \mu \neq\lambda, \mu, \lambda \ \textless \ \xi+1,$ and $\gamma \ \textless \ \xi$. Since $f$ is an injection we have that $f(\mu) \neq f(\lambda)$, so $p_{ \lambda \gamma }$ and $p_{ \mu \gamma}$ will never hold simultaneously at any world in $W$ by our definition of $V$. Hence, $M, t \Vdash p_{ \lambda \gamma } \wedge p_{ \mu \gamma} \rightarrow \bot$ by antecedent failure. However, $\Delta$ itself has no model, contradicting $\kappa$-compactness.

Suppose on the other hand that $\kappa$ is singular. In \cite{dick} (p. 85) it is noted that the infinitary languages $L_{\kappa \omega}$ where $\kappa$ is singular are exactly as expressive as  languages $L_{\kappa^+ \omega}$. The argument holds  for $L_{\kappa \omega}^{\rightarrow}$ as well.  Hence, without loss of generality, we can take 
\begin{center}
$\Delta = \{\bigvee_{\gamma \ \textless \ \kappa} p_{\lambda \gamma }: \lambda \ \leq \ \kappa\} \cup \{p_{ \lambda \gamma } \wedge p_{ \mu \gamma} \rightarrow \bot :\mu \neq\lambda, \mu, \lambda \ \leq \ \kappa, \gamma \ \textless \ \kappa \}$
\end{center}
to be a perfectly good collection of formulas of $L_{\kappa \omega}^{\rightarrow}$. As before every subset $\Delta_0 \subset \Delta$ such that $|\Delta_0| \ \textless \  \kappa$ has a model in $V_{\bf S}$ but $\Delta$ does not.
 \end{proof}

A Hilbert-style formal system $H$ for a language  $L_{\kappa \omega}^{\rightarrow}$ with respect to the class of   models  for a standard system for relevant logic will be formed by a set of formulas  of  $L_{\kappa \omega}^{\rightarrow}$ taken as the collection of axioms and a collection of rules of inference each with less than $\kappa$ premises. If $\Gamma$ is a collection of formulas of $L_{\kappa \omega}^{\rightarrow}$ and $\phi$ a formula of $L_{\kappa \omega}^{\rightarrow}$,  we will write $\Gamma \vdash_H \phi$ if there is a sequence of formulas $S$ of length less than $\kappa$ such that every formula in $S$ is either an axiom, one of the formulas in $\Gamma$ or it follows from  previous formulas in $S$ using one of the inference rules. 

\begin{Pro} Let $|\mathtt{PROP}|  \geqslant \kappa^+$. Let $H$ be a formal system for  $L_{\kappa^+ \omega}^{\rightarrow}$  sound with respect to some $V_{\bf S}$. Then $H$  is not strongly complete. \end{Pro}

\begin{proof} Take $\Delta$ in the proof of Proposition \ref{pro:in}. Since every $\Delta_0 \subseteq \Delta$ with $|\Delta_0|\ \textless \ \kappa$ has a model in $V_{\bf S}$, by the soundness of $H$, we see that $\Delta_0 \nvdash_{H} \bot$, but that means that $\Delta \nvdash_{H} \bot$. However $\Delta$ semantically implies $\bot$ over $V_{\bf S}$, since it has no model. \end{proof}

\section{Relevant directed bisimulations and  Karp's Theorem}\label{sec:db}

In this section, we introduce relevant directed bisimulations, establish some basic facts that  will be needed in \S \ref{sec:pre} and prove the relevant analogue of Karp's theorem. The present section as well as \S \ref{sec:pre} focuses on the  infinitary relevant language with absurdity $L^{\rightarrow}_{\infty \omega}$.

\begin{Def} \label{def:dg} The \emph{ degree} of an infinitary relevant formula $\phi$, in symbols, $dg(\phi)$, is defined inductively in the following way:

\begin{center}
%\begin{table}[h]
\begin{tabular}{llll}
\,\,\,\,\, \,\,\,\,\, \, \, \,\,\, $dg( \bot)$ & = &              \emph{0},                   \\
\,\,\,\,\, \,\,\,\,\, \, \, \,\,\, \ $dg( p)$ & = & \emph{0},                                \\
                                                       
                                                          \,\,\,\,\, \,\,\,\, \,    $dg(\bigwedge_{i \in I} \phi_i)$        &  =  &    $  \mbox{\emph{sup}}\{dg(\phi_i) : i \in I\}$, \\
\,\,\,\,\, \,\,\,\, \,  $dg(\bigvee_{i \in I} \phi_i)$        &=    &      $  \mbox{\emph{sup}}\{dg(\phi_i) : i \in I\}$,    \\
  \,\,\,\,\, \,\,\,\, \, \,\,\,\,  $dg( \lsim \phi)$           &= & $ dg(\phi)$,\\
\,\,\,\,\, \,\,\,\, \  $dg(\phi \rightarrow \psi)$           & = &  $ \mbox{\emph{sup}}\{dg(\phi), dg(\psi)\}+1$.\\

\end{tabular}
%\end{table}
 \end{center}
  \end{Def}

We will say that two formulas $\phi$ and $\psi$ are \emph{equivalent} if for any model $(M, w)$, $M, w \Vdash \phi$ iff $M, w \Vdash \psi$.

\begin{Pro}\label{pro:equiv} For each ordinal $\alpha$, there are only set-many non-equivalent formulas of $L^{\rightarrow}_{\infty \omega}$ with degree $\leq \alpha$. \end{Pro}

\begin{proof} Consider first $L^{corr}_{\infty \omega}$. Define the quantifier rank of a formula of $L^{corr}_{\infty \omega}$ following \cite{bar1} (Definition 10. 4) which deals appropriately with the presence of functions in the language. According to Corollary 10.9 in \cite{bar1}, for $\kappa$ some fixed point of the function $\beth$ with cardinality bigger than the cardinality of the signature of  $L^{corr}_{\infty \omega}$ (there is always some such $\kappa$ given that $\beth$ is normal), every formula of  $L^{corr}_{\infty \omega}$ with quantifier rank $\leq \alpha$ is equivalent to a disjunction of size smaller than $\kappa$ of formulas of a certain class $\Delta$ with fewer than $\kappa$ non-equivalent members. Clearly, there are only set-many non-equivalent such disjunctions. Hence, there are only set-many non-equivalent formulas of $L^{corr}_{\infty \omega}$ with quantifier rank $\leq \alpha$. 

Finally since relevant formulas of degree $\leq \alpha$ can be seen via the translation as formulas of $L^{corr}_{\infty \omega}$ with quantifier rank $\leq \beta$ for sufficiently big $\beta$, we have established the result. \end{proof}

Relevant directed bisimulations $-$as bisimulations in modal logic$-$ are ``non-classical" analogues of back and forth games from classical model theory. In this sense, the next definition introduces the analogue of Definition 5.3.3 from \cite{dick}. 

\begin{Def} \label{def:ndbt}
Let $M_1=\langle W_1, R_1, *_1, V_1 \rangle$ and $M_2=\langle W_2, R_2, *_2, V_2 \rangle$ be two models.  A \emph{relevant directed $\alpha$-bisimulation} for  $\mathtt{PROP}$ between  $M_1$ and $M_2$  is a system of pairs of non-empty relations $\langle Z_{01}, Z_{02}\rangle, \dots , \langle Z_{\alpha 1}, Z_{\alpha 2}\rangle$ where
\begin{center}
 $Z_{\beta 1} \subseteq W_1 \times W_2$ and  $Z_{\beta 2}  \subseteq  W_2 \times W_1$ \,\,\,\,\, $(0 \leq \beta \leq \alpha)$
\end{center}
such that 
\begin{center}

$Z_{\alpha 1} \subseteq \dots \subseteq Z_{01} $

$Z_{\alpha 2} \subseteq \dots  \subseteq  Z_{02}$ 

\end{center}
and when  $i, j \in \{1, 2\}$, $0 \leq  \beta \ \textless \ \alpha$ and $0 \leq \gamma \leq \alpha$,

\begin{itemize}
\item [(1)] $xZ_{\gamma i} y$ only if $y^{*_j}Z_{\gamma j} x^{*_i}$
\item [(2)] If $xZ_{(\beta + 1)i}y$ and $R_jybc$ for some $b, c \in W_j$, there are $b^{\prime}, c^{\prime} \in W_i$ such that $R_ixb^{\prime}c^{\prime}$, $bZ_{\beta j} b^{\prime}$ and $c^{\prime}Z_{\beta i} c$.

\item [(3)] If $xZ_{\gamma i} y$ and $p \in \mathtt{PROP}$,

\begin{center}
\begin{itemize}
\item[] $M_i, x \Vdash p$ only if $M_j, y \Vdash p$. 

\end{itemize}
\end{center}
\end{itemize}
\end{Def}

\begin{Pro} \label{pro:karp} Let $(M_1, w_1)$ and $(M_2, w_2)$ be two arbitrary Routley-Meyer models, $\alpha$ an ordinal and $i, j \in \{1, 2\}$. Then, (i) for each relevant formula $\phi$ of  $L^{\rightarrow}_{\infty \omega}$ with degree $\leq \alpha$, $M_i, w_i \Vdash \phi$ only if $M_j, w_j \Vdash \phi$ iff (ii) there is a relevant directed $\alpha$-bisimulation $(\langle Z_{\beta i}, Z_{\beta j} \rangle)_{\beta \leq \alpha}$ such that $w_{i} Z_{\beta j} w_{ j}$  for each $\beta \leq \alpha$.
\end{Pro}
\begin{proof}$ (ii) \Rightarrow (i)$: Assume that (ii). We argue for (i)  for all $\alpha$ simultaneously, by induction on the complexity of $\phi$. 

 The atomic cases as well as $\bot$  are obvious from (3) in Definition \ref{def:ndbt} and the fact that $\bot$ is never true.   For negation, let $\phi = (\lsim \psi)$ and suppose that $M_i, w_i  \Vdash (\lsim \psi) $, so $M_i, w_i^{*_i} \nVdash \psi $. But $ w_j^{*_j}  Z_{\alpha j} w_i^{*_i}$ by (1) in Definition \ref{def:ndbt} since $w_{i} Z_{\alpha j} w_{ j}$ by assumption, and, by inductive hypothesis, $M_j, w_j^{*_j} \nVdash \psi$, so $M_j, w_j \Vdash (\lsim\psi)$ as desired.  Conjunction and disjunction are routine exercises.

The only remaining case is when $\phi = \psi \rightarrow \chi$. By Definition \ref{def:dg}, say that $dg(\phi) = \beta + 1 \leq \alpha$ where $\beta = \mbox{sup}\{dg(\psi), dg(\chi)\}$. 
 Suppose that   $M_i, w_i \Vdash \psi \rightarrow \chi$, which means that if $R_i w_ib^{\prime}c^{\prime}$ for some $b^{\prime},c^{\prime}$, and $M_i, b^{\prime} \Vdash \psi $, then $M_i, c^{\prime} \Vdash \chi$. Now, let $R_j w_jbc$ for arbitrary $b, c$.  We need to show that $M_j, b \Vdash \psi$ only if $M_j, c \Vdash \chi$. To get the contrapositive, we will suppose that $M_j, c \nVdash \chi$.  By the assumption (ii), $w_iZ_{\beta + 1 i}w_j$, so using property (2) in  Definition \ref{def:ndbt}, there are $b^{\prime}, c^{\prime}$ such that $R_i w_ib^{\prime}c^{\prime}$, $bZ_{\beta j} b^{\prime}$ and $c^{\prime}Z_{\beta i} c$. Note that $\langle Z_{0i}, Z_{0j}\rangle, \dots , \langle Z_{\beta i}, Z_{\beta j}\rangle$ is a directed $\beta$-bisimulation between $M_i$ and $M_j$. This follows readily from our assumption that $\langle Z_{0i}, Z_{0j}\rangle, \dots , \langle Z_{\alpha i}, Z_{\alpha j}\rangle$ is a relevant directed $\alpha$-bisimulation between $M_i$ and $M_j$ by verifying (1)-(3) in Definition \ref{def:ndbt}.   By inductive hypothesis, since $M_j, c \nVdash \chi$ and $dg(\chi) \leq \beta$, $M_i, c^{\prime} \nVdash \chi$. Given that  $M_i, w_i \Vdash \psi \rightarrow \chi$, it must be that $M_i, b^{\prime} \nVdash \psi$. But by inductive hypothesis again using the fact that $bZ_{\beta j} b^{\prime}$ and $dg(\psi) \leq \beta$, $M_j, b \nVdash \psi$. Hence, $M_j, w_j \Vdash \psi \rightarrow \chi$.

$ (i) \Rightarrow (ii)$: For a model $S$, and world $w$  from $S$, we denote by  $rel_{\leq \gamma}$-$tp_S(w)$ the relevant type up to degree $\gamma$ of $w$, i.e., the set of all infinitary relevant formulas such that $S, w \Vdash \phi $  and $dg(\phi) \leq \gamma$.  We claim that, on the assumption that (i), the following system of relations defines a relevant directed  $\alpha$-bisimulation  between $M_i$ and $M_j$:

\begin{center}

 $x Z_{\beta i} y$ iff $rel_{\leq \beta}$-$tp_{M_i}(x)  \subseteq rel_{\leq \beta}$-$tp_{M_j}(y)$ ($0\leq \beta \leq \alpha$)($i \neq j$, $i, j \in \{1,2\}$).

\end{center}

Let us first note that  $Z_{\alpha m} \subseteq \dots \subseteq Z_{0m}$ ($m \in \{1,2\}$). By the asumption (i), $Z_{\alpha i}$ is non-empty, since $w_iZ_{0i} w_j$, but the latter also implies that  $w_j^{*_j} Z_{\alpha j} w_i^{*_i}$ as we will see below, so $Z_{\alpha j}$ is non-empty. Hence, $Z_{\alpha m} $ ($m \in \{1,2\}$) has to be non-empty.

 Let $0\leq \beta \leq \alpha$, $i, j \in \{1, 2\}$.  If   $x Z_{\beta i} y$, i.e., $rel_{\leq \beta}$-$tp_{M_i}(x)  \subseteq rel_{\leq \beta}$-$tp_{M_j}(y)$, we see that $rel_{\leq \beta}$-$tp_{M_j}(y^{*_j})  \subseteq rel_{\leq \beta}$-$tp_{M_i}(x^{*_i})$, i.e., $y^{*_j} Z_{\beta j}x^{*_i}$. It suffices to show that  if   $M_j, y^{*_j} \Vdash \phi $ then $M_i, x^{*_i} \Vdash \phi $ for every $\phi$ with $dg(\phi) \leq \beta$. We prove  the contrapositive. Suppose that $M_i, x^{*_i} \nVdash  \phi $, so $M_i, x \Vdash ( \lsim \phi) $ and since $dg(\lsim \phi) = dg(\phi)$ and $rel_{\leq \beta}$-$tp_{M_i}(x)  \subseteq rel_{\leq \beta}$-$tp_{M_j}(y)$, also $M_j, y \Vdash  (\lsim \phi) $. Consequently,  $M_j, y^{*^j} \nVdash \phi $ as we wanted. This takes care of (1) in  Definition \ref{def:ndbt}.

For clause (2) in   Definition \ref{def:ndbt}, suppose that $xZ_{(\beta +1) i}y \ (\beta +1 \leq \alpha)$, i.e., $rel_{\leq \beta +1}$-$tp_{M_i}(x)  \subseteq rel_{\leq \beta +1}$-$tp_{M_j}(y)$, and  $R_j ybc$ for some worlds $b, c$ from  $M_j$. Where Fmla($L^{\rightarrow}_{\infty \omega}$) stands for the class  of propositional relevant formulas of $L^{\rightarrow}_{\infty \omega}$, consider

\begin{center}
$nrel_{\leq \beta}$-$tp_{M_j}(c) = \{\phi :  M_j, c \nVdash \phi, \phi \in \mbox{Fmla}(L^{\rightarrow}_{\infty \omega}), dg(\phi) \leq \beta\}$.
\end{center}

By Proposition \ref{pro:equiv}  we see that $rel_{\leq \beta}$-$tp_{M_j}(b) $ as well as  $ nrel_{\leq  \beta}$-$tp_{M_j}(c)$   can be taken as sets.  It is clear that 
\begin{center}
 $M_j, y \nVdash \bigwedge rel_{\leq \beta}$-$tp_{M_j}(b)  \rightarrow \bigvee nrel_{\leq  \beta}$-$tp_{M_j}(c)$
\end{center}
 since  $R_j ybc$,   $M_j, b \Vdash \bigwedge rel_{\leq  \beta}$-$tp_{M_j}(b)$ but $M_j, c \nVdash \bigvee  nrel_{\leq  \beta}$-$tp_{M_j}(c)$. Observe that
\begin{itemize}
 \item[]$dg( \bigwedge rel_{\leq  \beta}$-$tp_{M_j}(b)  \rightarrow \bigvee nrel_{\leq  \beta}$-$tp_{M_j}(c)) =  \mbox{sup}\{dg( \bigwedge rel_{\leq  \beta}$-$tp_{M_j}(b) ), dg( \bigvee  nrel_{\leq  \beta}$-$tp_{M_j}(c))\} +1$,
\end{itemize}
  but
\begin{center} 
$dg( \bigwedge rel_{\leq  \beta}$-$tp_{M_j}(b) ) = \mbox{sup}\{dg(\delta) : \delta \in  rel_{\leq  \beta}$-$tp_{M_j}(b)\} \leq \beta$
\end{center}
 and 
\begin{center}
$ dg( \bigvee nrel_{\leq  \beta}$-$tp_{M_j}(c))= \mbox{sup}\{dg(\sigma): \sigma \in nrel_{\leq  \beta}$-$tp_{M_j}(c))\} \leq  \beta$, 
\end{center}
so
\begin{center}
 $dg( \bigwedge rel_{\leq  \beta}$-$tp_{M_j}(b)  \rightarrow \bigvee nrel_{\leq  \beta}$-$tp_{M_j}(c)) \leq \beta +1$.

\end{center}
 Thus, since $rel_{\leq \beta +1}$-$tp_{M_i}(x)  \subseteq rel_{\leq \beta +1}$-$tp_{M_j}(y)$,  contraposing,
\begin{center}
 $M_i, x \nVdash \bigwedge rel_{\leq  \beta}$-$tp_{M_j}(b)  \rightarrow \bigvee nrel_{\leq  \beta}$-$tp_{M_j}(c)$,
\end{center}
 which means that there are $b^{\prime}$ and $ c^{\prime}$   such that $R_i xb^{\prime}c^{\prime}$, $M_i, b^{\prime} \Vdash  \bigwedge rel_{\leq  \beta}$-$tp_{M_j}(b)$, and  $M_i, c^{\prime} \nVdash  \bigvee nrel_{\leq  \beta}$-$tp_{M_j}(c)$. Hence, $rel_{\leq  \beta}$-$tp_{M_j}(b) \subseteq rel_{\leq  \beta}$-$tp_{M_i}(b^{\prime}) $, i.e., $bZ_{\beta j}b^{\prime}$. On the other hand, we have that if $ M_j, c \nVdash \phi$ then $M_i, c^{\prime} \nVdash \phi$ whenever  $dg(\phi) \leq\beta$. Contraposing, $rel_{\leq \beta}$-$tp_{M_i}(c^{\prime})  \subseteq rel_{\leq\beta}$-$tp_{M_j}( c)$, i.e.,  $c^{\prime}Z_{\beta i}c$.

Condition (3)   in  Definition \ref{def:ndbt} follows given that atomic formulas have degree 0.

 \end{proof}

\begin{Def} \label{def:dbt}
Let $M_1=\langle W_1, R_1, *_1, V_1 \rangle$ and $M_2=\langle W_2, R_2, *_2, V_2 \rangle$ be two models.  A \emph{relevant directed bisimulation} for  $\mathtt{PROP}$ between  $M_1$ and $M_2$  is a pair of non-empty relations $\langle Z_{1}, Z_{2}\rangle$ where
\begin{center}
 $Z_{1} \subseteq W_1 \times W_2$ and  $Z_{ 2}  \subseteq  W_2 \times W_1$ 
\end{center}
such that 
when  $i, j \in \{1, 2\}$, 

\begin{itemize}
\item [(1)] $xZ_{ i} y$ only if $y^{*_j}Z_{j} x^{*_i}$
\item [(2)] If $xZ_{i}y$ and $R_jybc$ for some $b, c \in W_j$, there are $b^{\prime}, c^{\prime} \in W_i$ such that $R_ixb^{\prime}c^{\prime}$, $bZ_{j} b^{\prime}$ and $c^{\prime}Z_{ i} c$.

\item [(3)] If $xZ_{ i} y$ and $p \in \mathtt{PROP}$,

\begin{center}
\begin{itemize}
\item[] $M_i, x \Vdash p$ only if $M_j, y \Vdash p$. 

\end{itemize}
\end{center}
\end{itemize}
\end{Def}

Next we show an analogue of Karp's celebrated theorem characterizing  $L_{\infty \omega}$-equivalence  in  terms of partial isomorphisms. The corresponding result for modal logic is regarded as  a ``folklore'' theorem.

\begin{Thm} \label{thm:karp}  \emph{(Relevant Karp's Theorem)} Let  $(M_1, w_1)$ and $(M_2, w_2)$ be two models and $i, j \in \{1, 2\}$. Then the following are equivalent:

\begin{itemize}
\item[(i)] for every formula  $\phi$ of $L_{\infty \omega}^{\rightarrow}$, $M_i, w_i \Vdash \phi$ only if $M_j, w_j \Vdash \phi$

\item[(ii)] there is a relevant directed bisimulation $\langle Z_i, Z_j\rangle$ between $M_1$ and $M_2$ such that $w_i Z_i w_j$.

\end{itemize}

   \end{Thm}
\begin{proof} $(ii) \Rightarrow (i)$: This direction follows from Proposition \ref{pro:karp}, and the facts that $\langle Z_i, Z_j\rangle$ can be taken to be a relevant directed $\alpha$-bisimulation for any $\alpha$ and that every formula of   $L_{\infty \omega}^{\rightarrow}$  has some degree $\alpha$.

$(i) \Rightarrow (ii)$: We claim that 
 
\begin{center}

 $x Z_{ i} y$ iff $rel$-$tp_{M_i}(x)  \subseteq rel$-$tp_{M_j}(y)$ ($i \neq j$, $i, j \in \{1,2\}$).

\end{center}
defines a relevant directed bisimulation where $rel$-$tp_{M_i}(x)$ ($i=1, 2$) is the collection of all formulas of $L_{\infty \omega}^{\rightarrow}$ holding at $x$ in $M_i$.

 For clause (1) in Definition \ref{def:dbt}, suppose $x Z_{\beta i} y$, i.e., $rel$-$tp_{M_i}(x)  \subseteq rel$-$tp_{M_j}(y)$. We have that that $rel$-$tp_{M_j}(y^{*_j})  \subseteq rel$-$tp_{M_i}(x^{*_i})$, i.e., $y^{*_j} Z_{\beta j}x^{*_i}$. It suffices to show that  if   $M_j, y^{*_j} \Vdash \phi $ then $M_i, x^{*_i} \Vdash \phi $ for every $\phi$. We prove  the contrapositive. Suppose that $M_i, x^{*_i} \nVdash  \phi $, so $M_i, x \Vdash ( \lsim \phi) $ and since $rel$-$tp_{M_i}(x)  \subseteq rel$-$tp_{M_j}(y)$, also $M_j, y \Vdash  (\lsim \phi) $. Consequently,  $M_j, y^{*^j} \nVdash \phi $ as we wanted. 

Now we have to take care of clause (2) in Definition \ref{def:dbt}. Assume that $xZ_{ i}y $, i.e., $rel$-$tp_{M_i}(x)  \subseteq rel$-$tp_{M_j}(y)$, and  $R_j ybc$ for some worlds $b, c$ from  $M_j$.  Suppose for reductio that there are no $b^{\prime}, c^{\prime} \in W_i$ such that $R_ixb^{\prime}c^{\prime}$, $bZ_{j} b^{\prime}$ (i.e., $rel$-$tp_{M_j}(b)  \subseteq rel$-$tp_{M_i}( b^{\prime})$) and $c^{\prime}Z_{ i} c$ (i.e., $rel$-$tp_{M_i}(c^{\prime})  \subseteq rel$-$tp_{M_j}(c)$). We  first notice that  $\{ b^{\prime}, c^{\prime} \in W_i  : R_ix b^{\prime}c^{\prime} \} \neq \emptyset$, for otherwise $M_i, x \Vdash \top \rightarrow \bot$, so $M_j, y \Vdash  \top \rightarrow \bot$, which implies that $M_j, c \Vdash \bot$, which is impossible.  Now, for any $b^{\prime}, c^{\prime} \in W_i$ such that   $R_ixb^{\prime}c^{\prime}$ there are formulas $\phi_{b^{\prime}}$ and $\phi_{c^{\prime}}$ such that either (i) $M_j, b \Vdash \phi_{b^{\prime}}$ and $M_i, b^{\prime} \nVdash \phi_{b^{\prime}}$ or (ii) $M_i, c^{\prime} \Vdash \phi_{c^{\prime}}$ and  $M_j, c \nVdash \phi_{c^{\prime}}$.  For any $b^{\prime}, c^{\prime} \in W_i$ such that   $R_ixb^{\prime}c^{\prime}$ define the transformation $\tau$ as follows: 

\begin{center}
$
\tau (\phi_{b^{\prime}}) = \begin{cases} \top &\mbox{if (i) does not hold} \\
\phi_{b^{\prime}} & \mbox{otherwise}. \end{cases} $

$
\tau (\phi_{c^{\prime}}) = \begin{cases} \bot &\mbox{if (ii) does not hold} \\
\phi_{c^{\prime}} & \mbox{otherwise}. \end{cases} $

\end{center}
Next, it suffices to consider the formula 
\begin{center}
$\bigwedge_{\exists v Rxb^{\prime}v \atop{b^{\prime} \in W_i}} \tau (\phi_{b^{\prime}})  \rightarrow \bigvee_{\exists v R_ixvc^{\prime} \atop{c^{\prime} \in W_i}} 
\tau (\phi_{c^{\prime}})$.
\end{center}
 A moments reflection shows that
\begin{center}
 $M_j, y \nVdash \bigwedge_{\exists v R_ixb^{\prime}v \atop{b^{\prime} \in W_i}} \tau (\phi_{b^{\prime}})  \rightarrow \bigvee_{\exists v R_ixvc^{\prime} \atop{c^{\prime} \in W_i}} 
\tau (\phi_{c^{\prime}})$
\end{center}
 but 
\begin{center}
$M_i, x \Vdash \bigwedge_{\exists v R_ixb^{\prime}v \atop{b^{\prime} \in W_i}} \tau (\phi_{b^{\prime}})  \rightarrow \bigvee_{\exists v R_ixvc^{\prime} \atop{c^{\prime} \in W_i}} 
\tau (\phi_{c^{\prime}})$, 
\end{center}
contradicting the assumption that $rel$-$tp_{M_i}(x)  \subseteq rel$-$tp_{M_j}(y)$.

Finally, clause (3) in Definition \ref{def:dbt} is immediate.
 \end{proof}

Theorem \ref{thm:karp} is nothing but the infinitary version of Theorem 13.5 from \cite{restall}. Quite frequently in infinitary logic we are able to obtain counterparts to results  provable for finitary languages with the restriction that the models under consideration be finite.

\section{Scott's theorem}\label{sec:s}

Next we establish a result  implying a corollary  analogous  to Scott's isomorphism theorem in  classical infinitary logic.  The corresponding theorem for modal logic was proven in \cite{ber}.

Since the finitary relevant logic is considerably weaker than first order logic and modal logic in terms of expressive power, it only seems natural that to get a version of Scott's isomorphism theorem one has to go beyond the expressive power gained by  merely  adding countable conjunctions. In fact, Corollary \ref{scott} requires us to add conjunctions of cardinality at most $ |2^{\omega}|$.

There is another difference between the following result and Scott's isomorphism theorem or van Benthem's modal version of it. Scott's theorem gives a formula $\phi_M$ characterizing up to isomorphism a given countable model $M$ among the class of all countable models, so Scott's formula only depends on the model $M$. In contrast,  we give a formula that implies that there is a relevant directed bisimulation between two arbitrary countable models but which depends on both. This difference is due to the nature of relevant directed bisimulations. Contrary to isomorphism or bisimulation, a relevant directed bisimulation between $M_1$ and $M_2$  demands things from  both models. Recall that it is not a relation from $W_1 \times W_2$ but a pair of relations from  $W_1 \times W_2$ and  $W_2 \times W_1$ respectively. 

\begin{Thm}   Let  $(M_1, w_1)$ and $(M_2, w_2)$ be two  models in some $K$ such that $K \subseteq V_{\bf B}$, $\kappa $  the least infinite cardinal $\geqslant$  $\mbox{\emph{sup}}\{|W_1| , |W_2|\}$,  and $\lambda = \mbox{\emph{sup}}\{| \mathtt{PROP} |, 2^{\kappa}\}$.  Then, when  $i, j \in \{1, 2\}$, there is a  formula $\theta^{w_i}$ of $L_{\lambda^+\omega}^{\rightarrow}$ such that (1) $M_i, w_i \Vdash \theta^{w_i}$,  and  (2) $M_j, w_j \Vdash \theta^{w_i}$   iff there is a relevant directed bisimulation $(Z_i, Z_j)$ between $M_i$ and $M_j$ such that $w_iZ_iw_j$. \end{Thm}

\begin{proof} We start by defining for each world $a$ of $M_i$ the formula $ \phi^{\eta a}_{M_j}$ $-$simultaneously with  $ \phi^{\eta b}_{M_i}$ for $b \in W_j$$-$ by induction on the ordinal $\eta \ \textless \ \lambda^+$ as follows:

\begin{center}
%\begin{table}[h]
\begin{tabular}{llll}
$\phi^{0 a}_{M_j}$ & $=$&   the set of all literals  satisfied by $(M_i,  a)$,                               \\
$\phi^{\eta a}_{M_j}$ &$=$ &  $ \bigwedge_{\xi \ \textless \ \eta} \phi^{\xi a}_{M_j}$ if $\eta$ is a limit ordinal,                               \\
                                                           $\phi^{\eta + 1 \  a}_{M_j} $           & $=$ &   $\phi^{\eta a}_{M_j}  \wedge \bigwedge_{b \in W_j, \atop{X \subseteq W_i, \atop{M_i, a \Vdash \phi^{\eta b}_{M_i} \rightarrow   \bigvee_{d \in X} \phi^{\eta d}_{M_j}}}} \phi^{\eta b}_{M_i} \rightarrow   \bigvee_{d \in X} \phi^{\eta d}_{M_j} $ \\
                                                                     &       &      $\wedge  \bigwedge_{b \in W_j, \atop{X \subseteq W_i, \atop{M_i, a \Vdash (\lsim (\phi^{\eta b}_{M_i} \rightarrow   \bigvee_{d \in X} \phi^{\eta d}_{M_j}))}}} (\lsim(\phi^{\eta b}_{M_i} \rightarrow   \bigvee_{d \in X} \phi^{\eta d}_{M_j})) $.    \\                                                                                                                             
\end{tabular}
%\end{table}
 \end{center}

Observe that when $\gamma \ \textless \ \beta \ \textless \ \lambda^+$, 

\begin{center}

$M_j, a^{\prime} \Vdash \phi^{\beta a}_{M_j} $ implies that $M_j, a^{\prime} \Vdash \phi^{\gamma a}_{M_j} $.

\end{center}

This can be seen by induction on $\beta$. The case when $\beta = 0$ is true by antecedent failure. If $\beta = \eta + 1$, either $\eta = \gamma$ or $\gamma \ \textless \ \eta$. If the first, since 
\begin{center}

$M_j, a^{\prime} \Vdash \phi^{\gamma a}_{M_j}  \wedge \bigwedge_{b \in W_j, \atop{X \subseteq W_i, \atop{M_i, a \Vdash \phi^{\gamma b}_{M_i} \rightarrow   \bigvee_{d \in X} \phi^{\gamma d}_{M_j}}}} \phi^{\gamma b}_{M_i} \rightarrow   \bigvee_{d \in X} \phi^{\gamma d}_{M_j}$ implies that $M_j, a^{\prime} \Vdash \phi^{\gamma a}_{M_j} $,
\end{center}
we have that

\begin{center}

$M_j, a^{\prime} \Vdash \phi^{\beta a}_{M_j} $ implies that $M_j, a^{\prime} \Vdash \phi^{\gamma a}_{M_j} $.

\end{center}
If the second, since  
\begin{center}

$M_j, a^{\prime} \Vdash \phi^{\beta a}_{M_j} $ implies that $M_j, a^{\prime} \Vdash \phi^{\eta a}_{M_j} $,

\end{center}
and, by inductive hypothesis, 

\begin{center}

$M_j, a^{\prime} \Vdash \phi^{\eta a}_{M_j} $ implies that $M_j, a^{\prime} \Vdash \phi^{\gamma a}_{M_j} $,

\end{center}
we get what we needed.

Now let us define a map $f : W_1 \times W_2 \longrightarrow \lambda^+$ in the following way:
\begin{center}
$
f(a, a^{\prime}) = \begin{cases} \mbox{the least ordinal} \ \xi \ \textless \  \lambda^+ \ \mbox{such that}\ M_2, a^{\prime} \nVdash \phi^{\xi a}_{M_2}&\mbox{if there is some} \\
0& \mbox{otherwise}. \end{cases} $

\end{center}
Given that $| W_1 \times W_2 | = \kappa \ \textless \ \mbox{cf}( \lambda^+) = \lambda^+$, we see that there must be $\xi_0 \ \textless \ \lambda^+$ such that the range of $f$ is a subset of $\xi_0$. Consequently, for every $\beta$ such that $\xi_0 \ \textless \ \beta \ \textless \ \lambda^+$,
\begin{center} 
$M_2, a^{\prime} \Vdash \phi^{\xi_0 a}_{M_2}$ implies that $M_2, a^{\prime} \Vdash \phi^{\beta a}_{M_2}$,
\end{center}
for otherwise we have that there is an ordinal $\gamma$ with $\xi_0  \ \textless \ \gamma \leq \beta$ which is the smallest ordinal such that $M_2, a^{\prime} \nVdash \phi^{\gamma a}_{M_2}$, contradicting the fact that the range of $f$ is a subset of $\xi_0$.

Similarly, we define $g : W_2 \times W_1 \longrightarrow \lambda^+$ as
\begin{center}

$f(a, a^{\prime}) = \begin{cases} \mbox{the least ordinal} \ \xi \ \textless \  \lambda^+ \ \mbox{such that} \ M_1, a^{\prime} \nVdash \phi^{\xi a}_{M_1}&\mbox{if there is some} \\
0& \mbox{otherwise}, \end{cases} $

\end{center}
and obtain $\xi_1  \ \textless \  \lambda^+$ such that the range of $g$ is a subset of  $\xi_1 $. As before, 
for every $\beta$ such that $\xi_1 \ \textless \ \beta \ \textless \ \lambda^+$,
\begin{center} 
$M_1, a^{\prime} \Vdash \phi^{\xi_1 a}_{M_1}$ implies that $M_1, a^{\prime} \Vdash \phi^{\beta a}_{M_1}$.
\end{center}

Choose  $\xi$ to be  $\mbox{sup}\{\xi_0, \xi_1\}$.  By the above, when $\xi \ \textless \ \beta \ \textless \ \lambda^+$,
\begin{center}
$M_2, a^{\prime} \Vdash \phi^{\xi a}_{M_2}$ implies that $M_2, a^{\prime} \Vdash \phi^{\beta a}_{M_2}$,
\end{center}
and
\begin{center}
 $M_1, a^{\prime} \Vdash \phi^{\xi a}_{M_1}$ implies that $M_1, a^{\prime} \Vdash \phi^{\beta a}_{M_1}$.
\end{center}
We claim that the relations $uZ_1v$ iff $M_2, v \Vdash \phi^{\xi u}_{M_2}$  and   $uZ_2v$ iff $M_1, v \Vdash \phi^{\xi u}_{M_1}$ satisfy all clauses in Definition \ref{def:dbt}.

For  (1)  in Definition \ref{def:dbt},  we will show  by induction that  when $i, j \in \{1, 2\}$,  for all $\beta$, if $u$ is a  world of  $M_i$ and  $M_j, v \Vdash \phi^{\beta u}_{M_j}$ then $M_i, u^{*_i} \Vdash \phi^{\beta v^{*_j}}_{M_i}$. In particular,  if $uZ_iv$, i.e., $M_j, v \Vdash \phi^{\xi u}_{M_j}$ then $M_i, u^{*_i} \Vdash \phi^{\xi  v^{*_j}}_{M_i}$, i.e., $v^{*_j}Z_ju^{*_i}$.

 Let $\beta =0$, and assume that $M_j, v \Vdash \phi^{0 u}_{M_j}$. We need to show that every literal satisfied by $v^{*_j}$ at $M_j$ is also satisfied by $u^{*_i} $ at  $M_i$, that is: (a) $M_j, v^{*_j} \Vdash p$   only if $M_i, u^{*_i} \Vdash p$, and (b) $M_j, v^{*_j} \Vdash (\lsim p)$   only if $M_i, u^{*_i} \Vdash (\lsim p)$. To prove the contrapositive of (a) assume that $M_i, u^{*_i} \nVdash p$, so  $M_i, u \Vdash (\lsim p)$, but $M_j, v \Vdash \phi^{0 u}_{M_j}$, hence $M_j, v \Vdash (\lsim p)$, i.e, $M_j, v^{*_j} \nVdash p$. Now, for the contrapositive of (b)  assume that $M_i, u^{*_i} \nVdash (\lsim p)$, so $M_i, u^{*_i *_i} \Vdash  p$ but $ u^{*_i *_i} = u$, so $M_i, u\Vdash  p$. However, $M_j, v \Vdash \phi^{0 u}_{M_j}$, which implies that $M_j, v\Vdash  p$, i.e., $M_j, v^{*_j *_j}\Vdash  p$, hence 
$M_j, v^{*_j }\nVdash  (\lsim p)$ as desired.

If $\beta$ is a limit ordinal and  $M_j, v \Vdash \phi^{\beta u}_{M_j}$, then $M_j, v \Vdash \phi^{\gamma u}_{M_j}$ for all $\gamma \ \textless \ \beta$, and by inductive hypothesis, $M_i, u^{*_i} \Vdash \phi^{\gamma v^{*_j}}_{M_i}$ for all $\gamma \ \textless \ \beta$, which implies that $M_i, u^{*_i} \Vdash \phi^{\beta v^{*_j}}_{M_i}$.

If $\beta = \gamma + 1$ and $M_j, v \Vdash \phi^{\gamma +1 u}_{M_j}$, $M_j, v \Vdash \phi^{\gamma u}_{M_j}$, and by inductive hypothesis, $M_i, u^{*_i} \Vdash \phi^{\gamma v^{*_j}}_{M_i}$.  Recall that

\begin{center}
%\begin{table}[h]
\begin{tabular}{llll}
\\
                                                           $\phi^{\gamma + 1 \  v^{*_j}}_{M_i} $           & $=$ &   $\phi^{\gamma \  v^{*_j}}_{M_i}   \wedge \bigwedge_{b \in W_i, \atop{X \subseteq W_j, \atop{M_j, v^{*_j} \Vdash \phi^{\gamma b}_{M_j} \rightarrow   \bigvee_{d \in X} \phi^{\gamma d}_{M_i}}}} \phi^{\gamma b}_{M_j} \rightarrow   \bigvee_{d \in X} \phi^{\gamma d}_{M_i} $ \\
                                                                     &       &      $\wedge  \bigwedge_{b \in W_i, \atop{X \subseteq W_j, \atop{M_j, v^{*_j} \Vdash (\lsim (\phi^{\gamma b}_{M_j} \rightarrow   \bigvee_{d \in X} \phi^{\gamma d}_{M_i}))}}} (\lsim(\phi^{\gamma b}_{M_j} \rightarrow   \bigvee_{d \in X} \phi^{\gamma d}_{M_i})) $.    \\                                                                                                                             
\end{tabular}
%\end{table}
 \end{center}
Hence, it remains to show that (a)  $M_j, v^{*_j} \Vdash \phi^{\gamma b}_{M_j} \rightarrow \bigvee_{d \in X} \phi^{\gamma d}_{M_i} $ for some world $b$ of $M_i$ and $X \subseteq W_j$ only if   $M_i, u^{*_i} \Vdash \phi^{\gamma b}_{M_j} \rightarrow \bigvee_{d \in X} \phi^{\gamma d}_{M_i} $, and (b) $M_j, v^{*_j} \Vdash (\lsim (\phi^{\gamma b}_{M_j} \rightarrow \bigvee_{d \in X} \phi^{\gamma d}_{M_i})) $ for some world $b$ of $M_i$ and $X \subseteq W_j$ only if   $M_i, u^{*_i} \Vdash (\lsim ( \phi^{\gamma b}_{M_j} \rightarrow \bigvee_{d \in X} \phi^{\gamma d}_{M_i} )) $. These two follow similarly to (a) and (b) in the case when $\beta = 0$.

The proof of  (2)  in Definition \ref{def:dbt} requires us to notice first that for  $i \in \{1, 2\}$, $M_i, u \Vdash \phi^{\beta u}_{M_j}$ for all $\beta$. We argue by induction on $\beta$. The case $\beta = 0$ is trivial. If $\beta $ is a limit ordinal and, by inductive hypothesis,  $M_i, u \Vdash \phi^{\gamma u}_{M_j}$ for all $\gamma \ \textless \ \beta$, then clearly  $M_i, u \Vdash \phi^{\beta u}_{M_j}$. Finally let $\beta = \gamma +1$. By inductive hypothesis, $M_i, u \Vdash \phi^{\gamma u}_{M_j}$. But trivially both (a)  $M_i, u \Vdash \phi^{\gamma b}_{M_i} \rightarrow \bigvee_{d \in X} \phi^{\gamma d}_{M_j} $ for some world $b$ of $M_j$ and $X \subseteq W_i$ only if   $M_i, u \Vdash \phi^{\gamma b}_{M_i} \rightarrow \bigvee_{d \in X} \phi^{\gamma d}_{M_j} $, and (b) $M_i, u \Vdash (\lsim(\phi^{\gamma b}_{M_i} \rightarrow \bigvee_{d \in X} \phi^{\gamma d}_{M_j})) $ for some world $b$ of $M_j$ and $X \subseteq W_i$ only if   $M_i, u \Vdash (\lsim(\phi^{\gamma b}_{M_i} \rightarrow \bigvee_{d \in X} \phi^{\gamma d}_{M_j})) $. Hence, $M_i, u \Vdash \phi^{\beta u}_{M_j}$.

Now, suppose that  $uZ_iv$, i.e., $M_j, v \Vdash \phi^{\xi u}_{M_j}$, which implies that $M_j, v \Vdash \phi^{\xi +1 u}_{M_j}$ by choice of $\xi$. Assume further that $R_jvbc$ and consider the disjunction $\bigvee_{d \in X} \phi^{\xi d}_{M_j}$ where $d \in W_i$ is such that $M_j, c \nVdash \phi^{\xi d}_{M_j}$. By a previous observation, $M_j, b \Vdash \phi^{\xi b}_{M_i}$ and clearly $M_j, c \nVdash \bigvee_{d \in X} \phi^{\xi d}_{M_j}$, so $M_j, v \nVdash \phi^{\xi b}_{M_i} \rightarrow \bigvee_{d \in X} \phi^{\xi d}_{M_j}$. Hence, given that $M_j, v \Vdash \phi^{\xi +1 u}_{M_j}$, $M_i, u \nVdash \phi^{\xi b}_{M_i} \rightarrow \bigvee_{d \in X} \phi^{\xi d}_{M_j}$. Thus, there are $b^{\prime}, c^{\prime} \in W_i$ such that $R_iub^{\prime}c^{\prime}$, $M_i, b^{\prime} \Vdash \phi^{\xi b}_{M_i}$, i.e., $bZ_ib^{\prime}$ and $M_i, c^{\prime} \nVdash \bigvee_{d \in X} \phi^{\xi d}_{M_j}$. The latter means that if $d \in W_i$ and $M_j, c \nVdash \phi^{\xi d}_{M_j}$ then $M_i, c^{\prime} \nVdash \phi^{\xi d}_{M_j}$. Again by a previous observation $M_i, c^{\prime} \Vdash  \phi^{\xi c^{\prime}}_{M_j}$, so we see that $M_j, c \Vdash \phi^{\xi c^{\prime}}_{M_j}$ contraposing the previous sentence, i.e., $c^{\prime}Z_ic $. 

Clause (3)  in Definition \ref{def:dbt} follows as if  $i, j \in \{1, 2\}$ and $uZ_iv$, i.e., $M_j, v \Vdash \phi^{\xi u}_{M_j}$ then $M_j, v \Vdash \phi^{0 u}_{M_j}$, so every propositional variable satisfied at $u$ in $M_i$ is also satisfied at $v$ in $M_j$.

The right to left direction of the theorem follows since if $M_j, w_j \Vdash \phi^{\xi w_i}_{M_j}$ then $Z_1$ and $Z_2$ are both non-empty, so we have the required relevant directed bisimulation between $M_i$ and $M_j$. 

For the other direction if there is one such relevant directed  bisimulation $M_j, w_j \Vdash \phi^{\beta w_i}_{M_j}$ for all $\beta$, so in particular, $M_j, w_j \Vdash \phi^{\xi w_i}_{M_j}$. This can be be seen by recalling that for any $\beta$,  $M_i, w_i \Vdash \phi^{\beta w_i}_{M_j}$ and since $w_i Z_i w_j$ by assumption, $M_j, w_j \Vdash \phi^{\beta w_i}_{M_j}$  since all formulas of $L_{\lambda^+\omega}^{\rightarrow}$  are preserved under relevant directed bisimulations.
\end{proof}

\begin{Cor} \label{scott}  \emph{(Relevant Scott's Theorem)} Let   $(M_1, w_1)$ and $(M_2, w_2)$ be two models in some $K$ such that $K \subseteq V_{\bf B}$,  and suppose $L_{|2^{\omega}|^+\omega}^{\rightarrow}$ has at most $ |2^{\omega}|$ propositional variables.  Then, when $i, j \in \{1, 2\}$, there is a  formula $\theta^{w_i}$ of $L_{|2^{\omega}|^+\omega}^{\rightarrow}$ such that $M_j, w_j \Vdash \theta^{w_i}$   iff there is a relevant directed bisimulation $(Z_i, Z_j)$ between $M_i$ and $M_j$ such that $w_iZ_iw_j$. \end{Cor}

\section{Interpolation, preservation and Beth definability}\label{sec:pre}

In this section, following the analogous case for modal logic \cite{bar, van}, we obtain a preservation theorem for relevant infinitary formulas as a corollary to a generalized interpolation result. Interpolation theorems have a history of implying preservation results (some examples in  infinitary logic  can be found in \cite{keis}).

Let $M$ be a structure for a language $L_{\infty \omega}$, if $X \subseteq \mbox{dom}(M)$ and $X$ is closed under all the functions in the signature of $M$, then $[X]^M$ is the submodel obtained by restricting all the relations in the signature of $M$ to $X$. Note that if $X$ fails to be closed closed under the required functions, then $[X]^M$  is not defined.

\begin{Lem} \label{lem:rel} \emph{(Relativization Lemma)} Let  $L_{\infty \omega}$ be a language with a unary predicate $P$. Then for any formula $\phi(\overline{x})$ $L_{\infty \omega}$  not containing $P$ there is a first order formula $\phi^P$ such that if $M$ is a structure where $[P^M]^M$ is defined then for every sequence $\overline{a}$ of elements from  $[P^M]^M$,
\begin{center}
 $M \vDash \phi^P[\overline{a}]$ iff $[P^M]^M \vDash \phi[\overline{a}]$.
\end{center}\end{Lem}
 \begin{proof} This is just Theorem 5.1.1 from \cite{hodges}. \end{proof}

Given a language $L$, by $\Sigma^1_1(L)$ and $\Pi^1_1(L)$ we will mean the languages resulting from admitting, respectively,  second order existential quantifications in front of a formula of $L$ and second order universal quantifications in front of a formula of $L$.

\begin{Lem} \label{lem:b} If $L_{\infty \omega}$ has a signature containing a binary symbol $\textless$, $\phi(x)$ and $\psi$ are formulas of $L_{\infty \omega}$ and $\Sigma^1_1 (L_{\infty \omega})$ respectively such that for each ordinal $\alpha$  there is a model M such that $\textless^M$  is a linear ordering on $\phi(M)$ in order type $\geqslant \alpha$, then $\psi$ has a model $N$ such that $\textless^N$ is a linear ordering on $\phi(N)$ which is not well-ordered.  \end{Lem}

\begin{proof} This is essentially Theorem 11.5.4 in \cite{hodges} or Theorem 1. 8 in \cite{bar2}. \end{proof}

Lemma \ref{lem:b} is known as the property of the model-theoretic language $L_{\infty \omega}$ of being \emph{bounded}, a  substitute for compactness when establishing that a property is not expressible in $L_{\infty \omega}$ (\cite{hodges}, p. 581). It is a useful  property that can be seen to characterize    $L_{\infty \omega}$ in terms of expressive power  via a Lindstr\"om theorem (cf. \cite{bar1}).

Let $\langle R, S \rangle$ be a pair of binary relations between two structures $M_1$ and $M_2$, while $\phi$ and $\psi$ are formulas of $L_{\infty \omega}^{corr}$. Following \cite{bar, van} we say that \emph{$\phi$ implies $\psi$ along $\langle R, S \rangle$} if whenever $M_1RM_2$, $M_1 \vDash \phi$ only if $M_2 \vDash \psi$ and if $M_2SM_1$, $M_2 \vDash \phi$ only if $M_1 \vDash \psi$. This can be seen as a generalization of the usual notion of consequence (note that standard consequence is the case when $R$ and $S$ are the identity). When the relation in question is relevant directed bisimulations, $\phi$ implies $\psi$ along relevant directed bisimulations if when $\langle Z_1, Z_2 \rangle$ is a relevant directed bisimulation between two models $M_1$ and $M_2$, and $aZ_ib$ ($i, j \in \{1, 2\}$) for elements $a, b$ of the domains of $M_i$ and $M_j$ respectively, then $M_i \vDash \phi [a]$ only if $M_j \vDash \phi[b]$.

If $\phi$ is a formula of  $L_{\infty \omega}^{corr}$, we will write $\mathtt{PROP}_{ \phi}$ for the collection of predicates appearing in $\phi$ corresponding to propositional variables in $\mathtt{PROP}$.

\begin{Lem} \label{lem:o} Let $\phi, \psi$ be formulas of  $\Sigma^1_1(L_{\infty \omega}^{corr}), \Pi^1_1(L_{\infty \omega}^{corr})$ respectively. Suppose $\phi$  implies $\psi$ along relevant directed bisimulations for  $\mathtt{PROP}_{ \phi} \cap\mathtt{PROP}_{ \psi} $  over some class of Routley-Meyer structures  $K$ defined by some formula $\sigma$ of $L_{\infty \omega}^{corr}$.  Then there is an ordinal $\alpha$ such that for every $M, N \in K$ if $M \vDash \phi[w]$ and $u$ satisfies in $N$ all the infinitary relevant formulas of degree    $\leq \alpha$ satisfied by $w$ in $M$, then $N \vDash \psi[u]$.  \end{Lem}

\begin{proof} Suppose  for reductio that for each $\alpha$ there are $(M_{1}, w_{1})$ and $(M_{2}, w_{ 2})$ such that $M_{1} \vDash \phi[w_{1}]$ and  $M_{ 2} \nvDash \psi[w_{ 2}]$ while  $w_{2}$ satisfies in $M_{ 2}$ all the infinitary relevant formulas of degree    $\leq \alpha$ satisfied by  $ w_{1}$ in $M_{ 1}$. Hence, by Proposition \ref{pro:karp}, there is a relevant directed $\alpha$-bisimulation $(\langle Z_{\beta 1}, Z_{\beta 2} \rangle)_{\beta \leq \alpha}$ such that $w_{1} Z_{\beta 1} w_{ 2}$ for each $\beta \leq \alpha$.

 Suppose for simplicity that $\mathtt{PROP}_{ \phi} \cap\mathtt{PROP}_{ \psi}$   has a single non-logical symbol $p$. So the  correspondence language $L_{\infty \omega}^{corr}$ has signature $K=\{*, R, P, Q_0, Q_1, \dots \}$ where $Q_i$ ($i=0, 1, \dots$) are the predicates corresponding the propositional variables not in  $\mathtt{PROP}_{ \phi} \cap\mathtt{PROP}_{ \psi}$. Expand this signature by adding the set of symbols $\{U_1, U_2, \textless , O, B_1, B_2, I, G\}$, where $U_1, U_2, , B_1, B_2$ and $O$ are  unary predicates, $\textless$ and $I$ are binary predicates, while $G$ is a ternary predicate.

Consider the infinitary formula  $\bigwedge \Theta$, where $\Theta$ is the theory containing  the following formulas:

\begin{itemize}

\item []$\sigma^{U_1}$, $\sigma^{U_2}$

\item[] \textquotedblleft There are $x, y$ such that $U_1x,  U_2y, \phi^{U_1}x, \neg \psi^{U_2}y$ and for all $z, u$ such that   $Oz, B_1u$ and $Izu$, we have that $Guxy$ \textquotedblright
\item[] \textquotedblleft $\textless$ is a discrete total ordering with first and last elements \textquotedblright

\item[] \textquotedblleft $O$ is the field of $\textless$ \textquotedblright

\item[]  \textquotedblleft If $U_ix$, then $ U_ix^*$\textquotedblright \,\,\,\,\,\, ($i \in \{1,2\}$)

\item[]  \textquotedblleft If $U_ix$ and $ Rxyz$, then $U_iy$ and $U_iz$\textquotedblright \,\,\,\,\,\, ($i \in \{1,2\}$)

\item[]  \textquotedblleft If $B_iz, Ou, Iuz$ and $Gzxy$, then $ U_ix$ and $U_jy$\textquotedblright \,\,\,\,\,\, ($i \in \{1,2\}$)

\item[]  \textquotedblleft For all $z$ such that $Oz$, there is $u$ with $B_iu$ and $Izu$\textquotedblright \,\,\,\,\,\, ($i \in \{1,2\}$)

\item[]  \textquotedblleft If $B_iz, Ou, Iuz$ and $Gzxy$, then there is $v$ such that $B_jv,  Iuv$, and  $Gvy^*x^*$\textquotedblright \,\,\,\,\,\, ($i \in \{1,2\}$)

\item[]  \textquotedblleft If $B_iz, Ou, Iuz$ and $Gzxy$, then $Px$ only if $Py$\textquotedblright \,\,\,\,\,\, ($i \in \{1,2\}$)

\item[] \textquotedblleft If $U_ix, U_jy, U_jb, U_jc, Oz, Iuz, B_iz, Gzxy,  Rybc, Ov$ and $v \  \textless \ u$, then  there are $w, w^{\prime}$ such that $Ivw, Ivw^{\prime}, B_jw, B_iw^{\prime}$ and there are $ b^{\prime}c^{\prime}$ such that $U_ib^{\prime}, U_ic^{\prime}, Rxb^{\prime}c^{\prime}, Gw b b^{\prime}$ and $Gw^{\prime}c^{\prime}c$\textquotedblright \,\,\,\,\,\, ($i \in \{1,2\}$)

\end{itemize}

The last three classes of sentences described in our presentation of $\Theta$ are simply  restatements in first order logic of conditions  appearing in the definition of a directed $\alpha$-bisimulation.  

For each ordinal $\alpha$, $\bigwedge \Theta$ has a model $M_{\alpha}$ such that the ordering $\textless^{M_{\alpha}}$ on $O^{M_{\alpha}}$ has  order type $\geqslant \alpha$.  To see this consider $(M_{1}, w_{1})$ and $(M_{2}, w_{ 2})$ as given by our reductio assumption, that is, $M_{1} \vDash \phi[w_{1}]$ and  $M_{ 2} \nvDash \psi[w_{ 2}]$ while there is a relevant directed $\alpha$-bisimulation $(\langle Z_{\beta 1}, Z_{\beta 2} \rangle)_{\beta \leq \alpha}$ such that $w_{1} Z_{\beta 1} w_{ 2}$ for each $\beta \leq \alpha$.

 We can suppose without loss of generality that $W_1 \cap W_2 = \emptyset$ (if this is not the case already simply take isomorphic copies of $M_1$ and $M_2$ satisfying the proviso).  Let $M_{\alpha}$ be any model $M_3$ such that:

\begin{itemize}
\item[] $W_{3} = W_1 \cup W_2 \cup \alpha+1 \cup \{Z_{\beta i} : \beta \leq \alpha, i \in \{1, 2\}\}$,

\item[] $R_3 = R_1 \cup R_2$,

\item[] $*_3 = *_1 \cup *_2$,

\item[] $U_{i}^{M_3} = W_i$ \,\,\,\,\,\, ($i \in \{1,2\}$),

\item[] $P^{M_3} = P^{M_1} \cup P^{M_2}$,

\item[] $Q_i^{M_3} = Q_i^{M_1} \cup Q_i^{M_2}$\,\,\,\,\,\, ($i= 0, 1, \dots$),

\item[] $B_{i}^{M_3} = \{Z_{\beta i} : \beta \leq \alpha\}$ \,\,\,\,\,\, ($i \in \{1,2\}$),

\item[] $O^{M_3}= \alpha +1$,

\item[] $\textless^{M_3}$ is the natural ordering on $\alpha +1$,

\item[] $I^{M_3}\beta y$ iff $\beta \leq \alpha$ and $y = Z_{\beta i}$ for some $i \in\{1, 2\}$,

\item[] $G^{M_3}xab$ iff  $x = Z \in \{Z_{\beta i} : \beta \leq \alpha, i \in \{1, 2\}\}$ and $aZb$.
\end{itemize}
It follows that $M_3 \vDash \bigwedge \Theta$. The sentences $\sigma^{U_1}$, $\sigma^{U_2}$ hold in $M_3$ by Lemma \ref{lem:rel}, the fact that both $M_1$ and $M_2$ make $\psi$ true, and that $[U_1^{M_3}]^{M_3} = M_1$ and $[U_2^{M_3}]^{M_3} = M_2$.

Since for each ordinal $\alpha$, $\bigwedge \Theta$ has a model $M_{\alpha}$ such that the ordering $\textless^{M_{\alpha}}$ on $O^{M_{\alpha}}$ has  order type $\geqslant \alpha$, by Lemma \ref{lem:b}, $\bigwedge \Theta$ has a model $M_4$ such that $\textless^{M_4}$ is a linear ordering which is not well ordered. This means that $O^{M_4}$ being the field of $\textless^{M_4}$ contains an infinite descending sequence:

\begin{itemize}
\item[$(*)$]  $\dots e_3\ \textless^{M_4} \  e_2\ \textless^{M_4} \ e_1 \ \textless^{M_4} \ e_0.$
\end{itemize}

 Let $M_4|K$ be the restriction of $M_4$ to the signature $K$. Now, since $M_4$ makes $\bigwedge \Theta$ hold, there are $a \in U_{1}^{M_4} $ and $b \in U_{2}^{M_4} $ such that $M_4 \vDash \phi^{U_1}[a] $ (i.e., $[U_{1}^{M_4} ]^{M_4|K} \vDash  \phi[a] $),   $M_4 \nvDash  \psi^{U_2}[b] $ (i.e., $[U_{2}^{M_4} ]^{M_4|K} \nvDash  \psi[b] $) and for all $z, u$ such that   $z \in O^{M_4}, u \in B_1^{M_4}$ and $M_4 \vDash I[zu]$, we have that $M_4 \vDash G[uab]$.

 The pair $\langle Z_1, Z_2\rangle$   defines a relevant directed bisimulation for $\mathtt{PROP}_{ \phi} \cap\mathtt{PROP}_{ \psi} $ between $[U_{1}^{M_4} ]^{M_4|K}  $ and $[U_{2}^{M_4} ]^{M_4|K}$ where 
\begin{itemize}
\item [] $xZ_1y$ iff   there is $e_n$ ($n \in \omega$) in the sequence $(*)$ such that there is $u \in B_1^{M_4}, M_4 \vDash I[e_nu]$ and $M_4 \vDash G[uxy]$,

\item [] $xZ_2y$ iff   there is $e_n$ ($n \in \omega$) in the sequence $(*)$ such that there is $u \in B_2^{M_4}, M_4 \vDash I[e_nu]$ and $M_4 \vDash G[uxy]$.
\end{itemize}

 First note that $Z_1 \neq \emptyset \neq Z_2$. For all $u$ and arbitrary $e_n$ such that   $u \in B_1^{M_4}$ and $M_4 \vDash I[e_nu]$, we have that $M_4 \vDash G[uab]$, and given that there is such a $u$, we have that $aZ_1b$. But one of the formulas in $\Theta$ implies that there is also $v \in B_2^{M_4}$  such  $M_4 \vDash I[e_nv]$ and $M_4 \vDash G[vb^{*_4}a^{*_4}]$. Hence, $aZ_1b$ and $b^{*_4}Z_2a^{*_4}$, i.e., $b^{*_{[U_{2}^{M_4} ]^{M_4|K} }}Z_2a^{*_{[U_{1}^{M_4} ]^{M_4|K} }}$.

 To show (1) in Definition \ref{def:dbt} suppose that $i \in \{1, 2\}$ and $xZ_iy$. By  essentially the argument in the above paragraph it follows that $y^{*_{[U_{j}^{M_4} ]^{M_4|K} }}Z_jx^{*_{[U_{i}^{M_4} ]^{M_4|K} }}$.

For clause (2) in Definition  \ref{def:dbt}, suppose that  $i \in \{1, 2\}$ and $xZ_iy$, so  there is $e_n$ ($n \in \omega$) in the sequence $(*)$ such that there is $u \in B_i^{M_4}, M_4 \vDash I[e_nu]$ and $M_4 \vDash G[uxy]$. Now let $R_{[U_{j}^{M_4} ]^{M_4|K}}ybc$ for some $b, c \in U_{j}^{M_4}$, i.e., $R_4ybc$ by Lemma \ref{lem:rel}. But  since $e_{n+1} \ \textless \ e_n$, there is formula in $\Theta$ which implies that there are $w, w^{\prime}$ such that $M_4 \vDash I[e_{n+1}w], M_4 \vDash I[e_{n+1}w^{\prime}], w \in B_j^{M_4}, w^{\prime} \in B_i^{M_4}$ and there are $ b^{\prime}c^{\prime}$ such that $b^{\prime}, c^{\prime} \in U_i^{M_4}, R_4xb^{\prime}c^{\prime} $(so, by  Lemma \ref{lem:rel}, $R_{[U_{i}^{M_4} ]^{M_4|K}}xb^{\prime}c^{\prime}$)$, M_4 \vDash G[w b b^{\prime}]$ and $M_4 \vDash G[w^{\prime}c^{\prime}c]$ (hence  $bZ_j b^{\prime}$ and $c^{\prime}Z_i c$).

Condition (3)  in  Definition  \ref{def:dbt} follows as if  $i \in \{1, 2\}$ and $aZ_ib$,  there is formula in $\Theta$ implying that $M_4 \vDash P[a]$ only if $M_4 \vDash P[b]$, and, by the Lemma \ref{lem:rel}, $[U_{i}^{M_4} ]^{M_4|K}  P[a]$ only if $[U_{j}^{M_4} ]^{M_4|K}  P[b]$.

Finally, since  the pair $\langle Z_1, Z_2\rangle$   defines a relevant directed bisimulation for $\mathtt{PROP}_{ \phi} \cap\mathtt{PROP}_{ \psi}$ between $[U_{1}^{M_4} ]^{M_4|K}  $ and $[U_{2}^{M_4} ]^{M_4|K}$ with $aZ_1b$, $[U_{2}^{M_4} ]^{M_4|K} \nvDash \psi[b]$ and $[U_{1}^{M_4} ]^{M_4|K} \vDash \phi[a]$ we have a contradiction with the assumption that $\phi$ implies $\psi$ along relevant directed bisimulations. Also, $[U_{1}^{M_4} ]^{M_4|K}  $ and $[U_{2}^{M_4} ]^{M_4|K}$ are in the class of models $K$ since $\psi$ holds in both by Lemma \ref{lem:rel}.

\end{proof}

\begin{Thm} \label{thm:int} \emph{(Interpolation)} Let $\phi, \psi$ be  formulas of  $\Sigma_1^1(L_{\infty \omega}^{corr})$, $\Pi_1^1(L_{\infty \omega}^{corr})$  respectively  and $K$ a class of Routley-Meyer structures axiomatizable by some formula $\sigma$ of $L_{\infty \omega}^{corr}$. Then, $\phi$ implies $\psi$ along relevant directed bisimulations for  $\mathtt{PROP}_{ \phi} \cap\mathtt{PROP}_{ \psi} $  over $K$ iff  there is a relevant  interpolant $\theta$  for $\phi$ and $\psi$ over $K$ according to the standard consequence relation, with propositional variables in  $\mathtt{PROP}_{ \phi} \cap\mathtt{PROP}_{ \psi} $. \end{Thm}

\begin{proof}For the  right to left direction of the theorem suppose that there is a relevant infinitary interpolant $\theta$  for $\phi$ and $\psi$ over $K$ with propositional variables in  $\mathtt{PROP}_{ \phi} \cap\mathtt{PROP}_{ \psi} $. That  $\phi$ implies $\psi$ along relevant directed bisimulations for  $\mathtt{PROP}_{ \phi} \cap\mathtt{PROP}_{ \psi} $  over $K$ follows from Theorem \ref{thm:karp} and the fact that $\theta$  is an interpolant for $\phi$ and $\psi$ according to the usual consequence relation.

 For the converse, by Lemma \ref{lem:o}, we know that there is an ordinal $\alpha$ such that for every $M, N \in K$ if $M \vDash \phi[w]$ and $u$ satisfies in $N$ all the infinitary relevant formulas of degree    $\leq \alpha$ satisfied by $w$ in $M$, then $N \vDash \psi[u]$. Consider the disjunction $\bigvee_{M \vDash \phi[w]} (\bigwedge rel_{\leq \alpha}(M, w))$, where $rel_{\leq \alpha}$ is the set of all translations of formulas of $L_{\infty \omega}^{\rightarrow}$ of degree $\leq \alpha$ with propositional variables in $\mathtt{PROP}_{ \phi} \cap\mathtt{PROP}_{ \psi} $.  The class of all non-equivalent formulas of $L_{\infty \omega}^{\rightarrow}$ of degree $\leq \alpha$ is a set according to Proposition \ref{pro:equiv}. Thus,   $\bigvee_{M \vDash \phi[w]} (\bigwedge rel_{\leq \alpha}(M, w))$ is a perfectly good formula of $L^{corr}_{\infty \omega}$. This formula is the desired interpolant of $\phi$ and $\psi$. It is easy to see that $\phi$  implies  $\bigvee_{M \vDash \phi[w]} (\bigwedge rel_{\leq \alpha}(M, w))$, while the latter implies $\psi$ by choice of $\alpha$.   \end{proof}

\begin{Cor}  \emph{(Preservation)}  Let $\phi$ be a formula of  $L_{\infty \omega}^{corr}$  and $K$ a class of Routley-Meyer structures defined by some formula $\psi$ of $L_{\infty \omega}^{corr}$. Then, $\phi$ is preserved under directed bisimulations in $K$ iff  $\phi$ is equivalent to an infinitary relevant formula over $K$.

\end{Cor}

\begin{proof}Right to left follows from Theorem \ref{thm:karp}. For the converse, just  set $\phi = \psi$ in Theorem \ref{thm:int}.\end{proof}

\begin{Cor}  \emph{(Beth definability)}  Let $P$ be a unary predicate not  in $L_{\infty \omega}^{corr}$, $\phi(P)$  a formula of  $L_{\infty \omega}^{corr} \cup \{P\}$  and $K$ a class of Routley-Meyer structures defined by some formula $\psi$ of  $L_{\infty \omega}^{corr} \cup \{P\}$. Then the following are equivalent:

\begin{itemize}
\item[(i)] There is a relevant formula $\theta(x)$ of $L_{\infty \omega}^{corr}$ such that $\theta(x) \equiv Px$ is a logical consequence of $\phi(P)$ in the standard classical sense.

\item[(ii)] If $(M_1, w_1, P^{M_1})$ and $(M_2, w_2, P^{M_2})$ are models of $\phi(P)$ such that $\langle Z_1, Z_2 \rangle$ is a relevant directed bisimulation between the restrictions $(M_1, w_1)$ and $(M_2, w_2)$ of $(M_1, w_1, P^{M_1})$ and $(M_2, w_2, P^{M_2})$  to  $L_{\infty \omega}^{corr}$, then $\langle Z_1, Z_2 \rangle$   is a relevant directed bisimulation between$(M_1, w_1, P^{M_1})$ and $(M_2, w_2, P^{M_2})$.

\end{itemize}

\end{Cor}
\begin{proof} (i) $\Rightarrow$ (ii):  It suffices to show that when $(M_1, w_1, P^{M_1})$ and $(M_2, w_2, P^{M_2})$ are models of $\phi(P)$ such that $\langle Z_1, Z_2 \rangle$ is a relevant directed bisimulation between the restrictions $(M_1, w_1)$ and $(M_2, w_2)$ of $(M_1, w_1, P^{M_1})$ and $(M_2, w_2, P^{M_2})$  to  $L_{\infty \omega}^{corr}$, if $x \in P^{M_i}$ and $xZ_iy$ then $y \in P^{M_j}$. The result follows by the assumption  (i) and the easy direction of Proposition \ref{pro:karp}. 

 (ii) $\Rightarrow$ (i): It is enough to establish that  $\exists P(\phi(P) \wedge Px)$ implies $\forall P (\phi(P) \supset Px)$ along relevant directed bisimulations for  $\mathtt{PROP}_{ \exists P(\phi(P) \wedge Px)} \cap\mathtt{PROP}_{ \forall P (\phi(P) \supset Px)} $  over $K$, since then, by Theorem \ref{thm:int}, it follows that there is a relevant formula $\theta(x)$ of $L_{\infty \omega}^{corr}$ which is an interpolant  for $\exists P(\phi(P) \wedge Px)$ and $\forall P (\phi(P) \supset Px)$ over $K$ according to the standard consequence relation. Consequently,  (i) holds. 
\end{proof}

\section{Conclusion} \label{sec:con}

We have shown that many facts from the model theory of classical infinitary logic have analogues in the context of relevant logic and the Routley-Meyer semantics. In particular, versions of Karp's theorem and Scott's isomorphism theorem  can be obtained. Also, most infinitary relevant languages with absurdity are incompact, from which we can derive incompleteness of most Hilbert systems based on them  (in the sense of there being a semantic consequence of a certain set of formulas which cannot be deduced from the set in the formal system).

 We have also showed that the formulas of classical infinitary relevant logic corresponding to infinitary relevant formulas are exactly those preserved under relevant directed bisimulations. This was obtained as a consequence of a certain interpolation result, from which a Beth definability theorem followed as well.

\section*{Acknowledments}
We are grateful to two anonymous reviewers who provided very detailed comments pointing out several errors in a previous version of this paper. The editor, Dave Ripley, also offered helpful suggestions. We are indebted to Zach Weber and Patrick Girard for their constant encouragement. We wish to thank the audience at the Frontiers of Non-Classicality conference held in January 2016 at Auckland, in particular, Ed Mares, Shawn Standefer, Greg Restall and Jeremy Seligman. We acknowledge the support by the Austrian Science Fund
(FWF): project I 1923-N25 (\emph{New perspectives on residuated posets}). Finally, a Marsden Fund Grant awarded to Zach Weber by the Royal Society of New Zealand partially funded research for this paper.

\end{document}